\documentclass[12pt]{article}
\usepackage{fullpage,amssymb,amsmath,amsthm,hyperref,tkz-graph}
\usepackage{authblk}
\newcommand{\nats}{\mathbb{N}}
\newcommand{\ints}{\mathbb{Z}}
\newcommand{\rats}{\mathbb{Q}}
\newcommand{\re}{\mathbb{R}}
\newcommand{\cx}{\mathbb{C}}

\newcommand{\cB}{\mathcal{B}}

\newcommand{\bc}{\mathbf{c}}
\newcommand{\cC}{\mathcal{C}}

\newcommand{\be}{\mathbf{e}}

\newcommand{\FF}{\mathbb{F}}
\newcommand{\cG}{\mathcal{G}}

\newcommand{\cI}{\mathcal{I}}
\newcommand{\sI}{\mathsf{I}}
\newcommand{\sJ}{\mathsf{J}}
\newcommand{\cK}{\mathcal{K}}
\newcommand{\cL}{\mathcal{L}}

\newcommand{\cR}{\mathcal{R}}
\newcommand{\br}{\mathbf{r}}
\newcommand{\cS}{\mathcal{S}}
\newcommand{\cT}{\mathcal{T}}
\newcommand{\bw}{\mathbf{w}}
\newcommand{\bx}{\mathbf{x}}
\newcommand{\cZ}{\mathcal{Z}}
\newtheorem{thm}{Theorem}[section]
\newtheorem{lem}[thm]{Lemma}
\newtheorem{prop}[thm]{Proposition}
\newtheorem{cor}[thm]{Corollary}
\newtheorem{defn}[thm]{Definition}
\newtheorem{example}{Example}[section]
\newtheorem{remark}{Remark}[section]

\DeclareMathOperator\Rad{Rad}

\DeclareMathOperator\Jac{Jac}
\DeclareMathOperator\mult{mult}

\DeclareSymbolFont{bbold}{U}{bbold}{m}{n}
\DeclareSymbolFontAlphabet{\mathbbold}{bbold}

\newcommand{\bchi}{\boldsymbol{\chi}}

\title{A polynomial ideal associated to any $t$-$(v,k,\lambda)$ design\thanks{The second author's research is supported by an NSERC discovery grant.}}
\author[1]{William~J.~Martin}
\author[2]{Douglas~R.~Stinson}
\affil[1]{Department of Mathematical Sciences and Department of Computer Science,
Worcester Polytechnic Institute, 
100 Institute Road, 
Worcester, MA 01609,
USA}
\affil[2]{David R.\ Cheriton School of Computer Science, University of Waterloo,
Waterloo, Ontario, N2L 3G1, Canada}
\date{\today} 

\begin{document}
\maketitle

\begin{abstract}
We consider ordered pairs $(X,\cB)$ where $X$ is a finite set of size $v$ and $\cB$ is some collection 
of $k$-element subsets of $X$ such that every $t$-element subset of $X$ is contained in exactly $\lambda$
``blocks'' $B\in \cB$ for some fixed $\lambda$.  We represent each block $B$ by a zero-one vector $\bc_B$ of
length $v$ and explore the ideal $\cI(\cB)$ of polynomials in $v$ variables with complex coefficients which vanish on
the set $\{ \bc_B \mid B \in \cB\}$.  After setting up the basic theory, we investigate two parameters related to
this ideal: $\gamma_1(\cB)$ is the smallest degree of a non-trivial polynomial in the ideal $\cI(\cB)$ 
and $\gamma_2(\cB)$ is the smallest integer $s$ such that $\cI(\cB)$ is generated by a set of polynomials 
of degree at most $s$. We first prove the general bounds  $t/2 < \gamma_1(\cB) \le \gamma_2(\cB) \le k$. 
Examining important families of examples, we find that, for symmetric 2-designs and
Steiner systems, we have $\gamma_2(\cB) \le t$. But we expect $\gamma_2(\cB)$
to be closer to $k$ for less structured designs and we indicate this by constructing infinitely many 
triple systems satisfying $\gamma_2(\cB)=k$.
\end{abstract}

%%%%%%%%%%%%%%%%%%%%%%%%%%%%%%%%%%%%%%%%%%
\section{Introduction}
\label{sec:intro}

Let $X$ be a finite set of size $v$ and consider a $k$-uniform hypergraph $(X,\cB)$ with vertex
set $X$ and block (hyperedge) set $\cB$.  We aim to study polynomial functions on $X$ which vanish on each element of $\cB$
so that $\cB$ may be viewed as the variety of some naturally defined ideal of polynomials in $v$ variables. In order 
to do so, we identify each block $B \in \cB$ with a 01-vector $\bc_B$, with entries indexed by the elements of 
$X$, whose $i^{\rm th}$ entry is equal to one if $i\in B$ and equal to zero otherwise. In other words, $\bc_B$ is the
$B^{\rm th}$ column of the point-block incidence matrix, $A$, of the hypergraph.  
In this paper we work
in characteristic zero and consider the polynomial ring $\cR = \cx[x_1,\ldots, x_v]$. The \emph{evaluation map}
$$ \varepsilon: \cR \rightarrow \cx^\cB $$
given by $ \varepsilon (f)\left(B \right) =  f( \bc_B) $ is a ring homomorphism and its kernel is the ideal of all
polynomials in $v$ variables which evaluate to zero on each block of the hypergraph. Denoting this kernel 
by  $\sI$, we see that $\sI$ is the ideal  of the finite variety $\{ \bc_B \mid B \in \cB \}$ and we write $\sI=\cI(\cB)$. 
Our goal in this paper is 
to explore connections between this ideal $\sI$ and the hypergraph $(X,\cB)$. Our primary question 
involves the identification of good generating sets $\cG$ for $\sI$ based on the combinatorial structure of 
the design $(X,\cB)$.
We are not concerned here with the actual size of the generating set; in fact, we prefer a set of polynomials
preserved by the automorphism group of the design. By ``good'' here, we mean principally that that polynomials in the
generating set are all of low degree. But we also seek polynomials that shed light on properties of 
the design.

\subsection{Background and related work}
\label{sec:background}

We are aware of no previous work on the ideals of designs, but algebraic geometry has a long history in the theory of error-correcting codes. We also have a theory of spherical codes and designs that involves the evaluation of multivariate polynomials at finite sets of points on the unit sphere. M\"{o}ller \cite{moller} constructs good quadrature 
rules for spherical integration by choosing zero sets of well-chosen families of polynomials.  Conder and 
Godsil \cite{condergodsil} studied the symmetric group as a polynomial space. The standard module of the 
Johnson association scheme $J(v,k)$ can be identified with polynomials of degree at most $k$ in $v$ variables 
in such a way that the polynomials of a given maximum degree $j$ correspond to a sum of eigenspaces 
$V_0+V_1+\cdots+V_j$.  This approach, which is implicit in \cite{del} is worked out in more detail in a later
paper of Delsarte \cite{del-hahn} (see also \cite{caldel}). These phenomena motivated Martin and Steele 
\cite{leech} to consider the ideal of polynomials that vanish on the shortest vectors of certain lattices.
In work in progress, Martin \cite{billideals} extends this to attach an ideal to any cometric association scheme by
viewing the columns of the $Q$-polynomial generator $E_1$ in the Bose-Mesner algebra as a finite variety. 
Several important examples are connected to combinatorial $t$-designs and -- when disentangled from the 
language of association schemes -- the problem of determining generating sets for the ideals of those 
association schemes reduces to the problem we address here. 

The ``polynomial method'' in combinatorics \cite{tao} is a powerful tool for deriving bounds on the size of 
combinatorial objects and for proving non-existence of extremal objects. In particular, a recent breakthrough by
Croot, Lev and Pach \cite{clp-annals} (see also \cite{ellenberg}) has stimulated interest in collections of multivariate
polynomials that vanish on some configuration in a finite vector space. Here, we work in characteristic zero and 
are interested in how the ideal we obtain is related to the structure of the design. 
By contrast, the authors of \cite{clp-annals,ellenberg} work over fields of positive characteristic. To see the 
connection, we note that, since our zero set lies in $\ints^v$, the polynomial generators $g\in\cG$ may always be chosen from  $\ints[\bx]$. So application of the reduction  $\ints \rightarrow \ints_p$ maps the ideal 
$\langle \cG\rangle \subseteq \ints[\bx]$  into the ideal of the same  finite variety considered modulo $p$.  
It will be an interesting follow-up task to determine when the image under  this map is the full ideal in positive characteristic.

\subsection{The trivial ideal}
\label{sec:trivial}

We pause to introduce our notation for the standard operations of elementary algebraic geometry. (See, e.g., 
\cite[Chapter~15]{df} for a basic introduction.) For a set $\cS \subseteq \cx^v$, we let $\cI(\cS)$ denote the 
ideal of all polynomials in $\cx[\bx]:=\cx[x_1,\ldots, x_v]$ that vanish at each point in $\cS$.
And if $\cG$ is any set of polynomials in $\cx[x_1,\ldots, x_v]$, we denote by $\cZ(\cG)$ the zero set of $\cG$, 
the collection of all points $\bc$ in $\cx^v$ which satisfy $f(\bc)=0$ for all $f \in \cG$. Note that, when $\cS$ is 
finite, we have $\cZ(\cI( \cS))= \cS$. In the opposite direction, for any ideal $\sJ$ of polynomials, the 
Nullstellensatz (see, e.g., \cite[p21]{fulton}, \cite[p179]{iva}) informs us that 
$\cI(\cZ(\sJ))=\Rad(\sJ)$, where $\Rad(\sJ)$ denotes the {\em radical} of ideal $\sJ$,  the
ideal of all polynomials $g$ such that $g^n\in \sJ$ for some positive integer $n$. A \emph{radical ideal}
is an ideal which is already closed under this process: $\Rad(\sJ) = \sJ$.

Our first example to consider is the \emph{complete uniform hypergraph} $(X,\cK^v_k)$.

\begin{lem}
\label{lem:idealofk-sets}
Let $X$ be a finite set of size $v\ge k$ and let $\cK^v_k = \binom{ X }{k}$ consist of all $k$-subsets of $X$.
Let 
\begin{equation}
\label{eqn:Gtrivial}
 \cG_0 = \{ x_1 + \cdots + x_v - k \} \cup \{ x_i^2-x_i \mid 1\le i\le v\}. 
\end{equation}
Then $\cI(\cK^v_k) = \langle \cG_0 \rangle$ and $\cZ( \langle \cG_0 \rangle) = \{ \bc_B \mid B \in \cK^v_k \}$. 
\end{lem}

\begin{proof}
One easily checks that each $\bc_B$ is a common zero of the polynomials in $\cG_0$. Conversely, any point 
in $\cx^v$ which is a zero of each polynomial in $\cG_0$ must be a 01-vector with exactly $k$ ones. In order
to verify that $\cG_0$ generates the full ideal, we check that the Zariski tangent space at each point is full-dimensional.
Let $B \subset X$ be a $k$-set; evaluating the gradient $\nabla f$ of $f(\bx)= x_j^2 - x_j$ at $\bc_B$ we obtain 
$\pm \be_j$ where $\be_j$ is the standard basis vector with a one in position $j$ and all other entries zero.
So the Jacobian of the set  $\cG_0$
of $v+1$  polynomials in $v$ variables evaluated at $\bc_B$ takes the form
$$ \Jac(\cG_0, \bc_B) = \left[ \left. \frac{ \partial f_i }{ \partial x_h} \middle| \raisebox{-10pt}{${\bc_B}$} \right.  \right]_{h,i} 
= \left[ \begin{array}{rrrrr}
1 & \pm 1 & 0 & \cdots & 0 \\
1 &  0 & \pm 1 & \cdots & 0 \\
\vdots &   &   & \ddots &  \\
1 &  0 & 0 & \cdots & \pm 1 \end{array} \right]  $$
which clearly has column rank equal to $v$. This guarantees that each $\bc_B$ is a simple zero of 
$\langle \cG_0\rangle$ and so the ideal is indeed radical. It now follows from the Nullstellensatz that 
$$\cI(   \{ \bc_B \mid B \in \cK^v_k \} ) = \cI( \cZ( \langle \cG_0 \rangle) ) = \Rad(  \langle \cG_0 \rangle) =  \langle \cG_0 \rangle. $$
See Section \ref{sec:radical} for details of these last calculations.
\end{proof}

%%%%%%%%%%%%%%%%%%%%%%%%%%%%%%%%%%%%%%%%%%
\section{Two parameters}
\label{sec:twopars}

In this paper, once we fix $v$ and $k$, every ideal will contain the ideal we have just considered. We call this 
the \emph{trivial ideal} and denote
\begin{equation}
\label{eqn:Trivial}
\cT = \left\langle  \cG_0 \right\rangle 
= \left\langle x_1 + \cdots + x_v - k , \ \ x_1^2 - x_1, \ldots, x_v^2- x_v \right\rangle . 
\end{equation}
For any $k$-uniform hypergraph $(X,\cB)$ on $v$ points, the ideal $\cI(\cB):=\cI( \{ \bc_B \mid B \in \cB\})$ 
will contain $\cT$ and  a polynomial  $f \in \cI(\cB)$ will be called \emph{non-trivial} if 
$f \not\in \cT$ and \emph{trivial} otherwise.

To each $C \subseteq \{1,2,\ldots,v\}$ we associate the monomial $x^C = \prod_{j\in C} x_j $ and note that,
for a block $B\in \cB$, the value of $x^C$ at point $\bc_B$ is one if $C \subseteq B$ and zero otherwise.
A $k$-uniform hypergraph $(X,\cB)$ is a $t$-$(v,k,\lambda)$ \emph{design} (or a block design of \emph{strength}
$t$) if, for every $t$-element subset $T\subseteq X$ of points, there are exactly $\lambda$ blocks $B\in \cB$ with
$T \subseteq B$ (so $\sum_{B\in \cB} f(\bc_B)=\lambda$ whenever $f(\bx)=x^T$ for some $t$-set $T\subseteq X$). Every $t$-$(v,k,\lambda)$ design is an $s$-$(v,k,\lambda_s)$
design for each $s\le t$ where $\lambda_s \binom{k-s}{t-s} = \lambda \binom{v-s}{t-s}$. 

The following characterization of $t$-designs is well-known (see, for example, Godsil \cite[Cor.~14.6.3]{godsil}).

\begin{lem}[Cf.~Delsarte {\cite[Theorem~7]{del-hahn}}]
\label{lem:characterizet-designs}
Let $X$ be a set of size $v$ and let $(X,\cB)$ be a $k$-uniform hypergraph defined on $X$ with corresponding 
vectors $\bc_B$ ($B\in \cB$) as defined above. Then $(X,\cB)$ is a $t$-design on $X$ if and only if the 
average over $\cB$ of any polynomial $f(\bx)$ in $v$ variables of total degree at most $t$ is equal  to
the average of $f(\bx)$ over the complete uniform hypergraph $\cK^v_k$ defined on $X$.
\end{lem}

\begin{proof}
Let $C\subseteq X$ with $|C| = s \le t$. Exactly $\binom{v-s}{k-s}$ elements of $\cK^v_k$ contain $C$ so 
the average value of $f(\bx) = x^C$ over $\{ \bc_B \mid  B \in \cK^v_k \}$ is 
$$\binom{v-s}{k-s} \bigg/ \binom{v}{k} = \frac{ k(k-1) \cdots  (k-s+1) }{ v(v-1) \cdots (v-s+1) }  = \frac{\lambda_s}{\lambda_0}$$
which is exactly the average of $f(\bx)$ over the block set $\cB$. So the result holds for monomials. But every
polynomial in $v$ variables of total degree at most $t$ is a linear combination of such monomials, so the result 
holds for these as well by linearity.
\end{proof}

Given $(X,\cB)$, we seek combinatorially meaningful generating sets for $\cI(\cB)$.
Two polynomials $f(\bx)$ and $g(\bx)$ have the same value at every point $\bc_B$, $B \in \cB$ if and only
if their difference belongs to the ideal $\cI(\cB)$. For example, $f(\bx)=x_j^2$ and $g(\bx)=x_j$ take the same 
value on every 01-vector so $x_j^2-x_j$ belongs to the ideal of any design. 
We say $f(\bx)$ is a \emph{multilinear polynomial} in $v$ variables  if $f$ is linear in each $x_i$: i.e., each 
monomial with non-zero coefficient in $f$ is a product of distinct indeterminates. Modulo the trivial 
ideal $\cT$, each polynomial in $\cx[x_1,\ldots, x_v]$ is equivalent to some (not necessarily unique) 
multilinear polynomial  with zero constant term.  With a preference for polynomials
of smallest possible degree, we define two fundamental parameters.

\begin{defn}
\label{def:gamma12}
Let $(X,\cB)$ be a non-empty, non-complete $k$-uniform hypergraph on vertex set $X=\{1,\ldots,v\}$ with corresponding ring of 
polynomials $\cR = \cx[x_1,\ldots, x_v]$. Let $\cI(\cB)$ and $\cT$ be defined as above. Define
$$ \gamma_1(\cB) = \min \left\{ \deg f \mid f \in \cI(\cB), \ f \not\in \cT \right\} $$
and
$$ \gamma_2(\cB) = \min \left\{  \max \{ \deg f : f \in  \cG \} \mid  \cG \subseteq \cR, \ \langle \cG \rangle = 
\cI(\cB)  \right\} .$$
\end{defn}

So $\gamma_1(\cB)$ is the smallest possible degree of a non-trivial polynomial that vanishes on each block and 
$\gamma_2(\cB)$ is the smallest integer $s$ such that $\cI(\cB)$ admits a generating set all polynomials of which
have degree at most $s$.  Obviously, $ \gamma_1(\cB) \le \gamma_2(\cB)$; designs satisfying equality here are
particularly interesting.

\begin{thm}
\label{thm:tover2}
If $(X,\cB)$ is a $t$-design ($t\ge 2$) and $f \in \cI(\cB)$ is non-trivial, then $\deg f > t/2$. So, 
for any non-trivial $t$-design $(X,\cB)$, $\gamma_1(\cB) \ge  (t+1)/2$. 
\end{thm}

\begin{proof}
Suppose $F \in \cI(\cB)$ has degree at most $t/2$. Write $F(\bx)=f(\bx)+ig(\bx)$ where $f,g\in \re[\bx]$ each have degree at most $t/2$. Since the entries of each $\bc_B$ are real, it's clear that $f,g\in \cI(\cB)$. 
Then $f^2 \in \cI(\cB)$ is a non-negative polynomial
of degree at most $t$. By Lemma \ref{lem:characterizet-designs}, its average over $\cB$ is zero hence its average over $\cK^v_k$ is also zero.
Since $f^2$ is everywhere non-negative, it must evaluate to zero on the incidence vector $\bc_B$ of
every $k$-set $B$. So it belongs to the ideal $\cI( \cK^v_k)$. Since this ideal is radical and contains 
$f^2$, it also contains $f$. By Lemma \ref{lem:idealofk-sets}, $f$ must be trivial. The same argument applies to $g$ and, hence, to $F$.
\end{proof}

\begin{remark}
The same sort of reasoning used in this proof shows that $\cI(\cB)$ admits a vector space basis, even a generating set, of polynomials with integer coefficients. Let $F$ be a polynomial which vanishes on $\cB$ and let $\{\zeta_1\ldots,\zeta_m\} \subseteq \cx$ be a basis for the subspace of $\cx$, as a vector space over $\rats$, that contains all the 
coefficients of $F$. Then there exist unique polynomials $F_1,\ldots,F_m$ in $\rats[\bx]$ with $F=\sum_h \zeta_h F_h$. Since  each $\bc_B$ is a vector with integer entries, the fact that $F$ evaluates to zero at $\bc_B$ implies
that each $F_h$ also vanishes at that point. So, scaling appropriately, we may assume each generator
belongs to $\ints[\bx]$.
\end{remark}

A standard result in the theory of designs (see, Cameron \cite{cameronnear} and Delsarte 
\cite[Theorem~5.21]{del}) is the fact that a $t$-design with $s$ distinct block intersection sizes satisfies $t\le 2s$.
We now show that Theorem \ref{thm:tover2} implies a stronger result, which we believe is new.

Let $C \subseteq X$ with characteristic vector $\bc$ and suppose $\{ |C \cap B| : B \in \cB \} = \{i_1,\ldots,i_s\}$. 
Then every $\bc_B$ for $B\in \cB$ is a zero of the degree $s$ \emph{zonal polynomial} 
$$ F(\bx) = ( \bc \cdot \bx - i_1 ) \cdots  ( \bc \cdot \bx - i_s ) . $$
Of course, if $|C|$ is sufficiently small, this polynomial belongs to the trivial ideal.

\begin{cor}
\label{cor:degree}
Let $(X,\cB)$ be a $t$-design and let $C_1,\ldots,C_\ell \subseteq X$ and $\{i_{1,1},\ldots, i_{1,s_1}, \ i_{2,1}, \ldots,$ $i_{\ell, s_\ell} \}$ be a multiset of integers such that, for every $B\in \cB$ there exist $1\le h\le \ell$ and $1\le j\le s_h$
with $|B\cap C_h| = i_{h,j}$. If there exists some $k$-set $S \not\in \cB$ with $|S \cap C_h| \not\in 
\{ i_{h,1},\ldots, i_{h,s_h} \}$ for all $h=1,\ldots,\ell$, then $\gamma_1(\cB) \le s_1+ \cdots  + s_\ell$ hence $s_1+ \cdots  + s_\ell > t/2$.
\end{cor}

\begin{proof}  For $Y \subseteq X$, define 01-vector $\bchi_{Y}$ by $(\bchi_Y)_a = 1$ if $a\in Y$ and 
$(\bchi_S)_a = 0$ otherwise. E.g., $\bchi_B = \bc_B$ when $B$ is a block.
Consider the product of $\ell$ zonal polynomials
$$ F(\bx) = \prod_{h=1}^\ell \prod_{j=1}^{s_h} \left( \bchi_{C_h} \cdot \bx  - i_{h,j} \right).$$
By hypothesis, $F(\bc_B)=0$ for every $B\in \cB$.  Since $F(\bchi_S)\neq 0$, $F$ is non-trivial. 
So, by Theorem  \ref{thm:tover2}, $\deg F > t/2$.
\end{proof}

\begin{lem}
\label{lem:gam1upper}
Let $(X,\cB)$ be a $t$-$(v,k,\lambda)$ design. Let $s$ denote the smallest integer such that 
$\binom{v}{s} > | \cB |$. Then $\gamma_1(\cB) \le s$.
\end{lem}

\begin{proof}
The vector space of functions on $\cK^v_k$ representable by multilinear polynomials in $\cR$ with zero 
constant term and total degree at most $s$ has dimension $\binom{v}{s}$.  
For the chosen value of $s$, there exists a non-zero multilinear polynomial $f(\bx) \in \cR$, 
of total degree at most $s$,  
which vanishes on each element of $\cB$. (We have $|\cB|$ equations and $\binom{v}{s}$ unknowns.) 
Being multilinear with zero constant term, $f(\bx)$ 
is non-trivial with degree at most $s$. 
\end{proof}

We finish this section with two instructive examples.

\begin{example}
Let us construct the ideal of the Fano plane. Let $X = \ints_7$ and take 
$$\cB = \left\{   \{  0,1,3 \},\{  1,2,4  \},\{  2,3,5  \},\{  3,4,6  \},\{  4,5,0  \},\{  5,6,1  \},\{  6,0,2   \}  \right\}.$$
Starting from $\cG_0$, let us build up a meaningful generating
set for $\cI(\cB)$. The unique triple in $\cB$ containing both $0$ and $1$ also contains $3$; this 
combinatorial condition may be encoded as $x_0x_1 - x_0x_1x_3 \in \cI(\cB)$. Alternatively,
including the quadratic polynomial $x_0x_1 - x_0x_3$ in a generating set $\cG$ for our ideal 
also guarantees that any vector $\bc \in \cZ(\langle \cG \rangle )$ with $\bc_0=1$ and $\bc_1=1$
must have $\bc_3=1$ as well. Up to sign, there are $\binom{7}{2}$ quadratic generators of this form and these,
together with those in $\cG_0$, generate the full ideal. 
\end{example}

\begin{example}
With $X=\{1,\ldots,9\}$ and $\cB=\left\{   \{ 1,2,3 \},\{ 4,5,6 \},\{ 7,8,9 \} \right\}$, a 1-design, we find generating set 
$$ \cG = \cG_0 \cup \left\{ x_1 - x_2, \ x_1-x_3, \ \ x_4-x_5, \ x_4-x_6, \ \ x_7-x_8, \ x_7-x_9 \right\}.$$
While we will not see further examples where the ideal is generated by $\cG_0$ and linear polynomials, we
will see that the difference of two monomials appears again as a useful tool.
\end{example}

%%%%%%%%%%%%%%%%%%%%%%%%%%%%%%%%%%%%%%%%%%
\section{Radical ideals}
\label{sec:radical}

In this section, we deal with a technicality which arises as we compute ideals of finite sets. We show 
here that every ideal containing the trivial ideal is radical, thereby eliminating any further need to check 
this property. 

Given a finite set
$\cS$ of points in $\cx^v$, it is often easy to come up with polynomials that vanish at each of those points and, 
with  a bit of work,  we might find a generating set $\cG$ for some ideal $\sJ = \langle \cG \rangle$ whose 
zero set is exactly $\cS$: $\cZ( \langle \cG \rangle )  = \cS$. Hilbert's Nullstellensatz then tells us that
\begin{equation}
\label{eqn:0satz}
 \cI ( \cS ) = \cI ( \cZ (\sJ ) ) = \Rad( \sJ ),
\end{equation}
the \emph{radical} of of ideal $\sJ$ given by 
$$ \Rad( \sJ) = \left\{  f \in \cR \mid \left( \exists n \in \nats \right) \left( f^n \in \sJ \right) \right\} . $$
The ideal $\sJ$ is a \emph{radical ideal} if $\Rad( \sJ) = \sJ$. Our goal then is achieved in three steps: 
given a finite set of points $\cS$,
find a nice set $\cG$ of small-degree polynomials that vanish on $\cS$; verify that $\cZ(\cG) = \cS$ and nothing 
more; verify also that $\langle \cG\rangle$ is a radical ideal. In this section, we discuss ways to achieve this last step.

If $\sJ$ is an ideal in $\cx[x_1,\ldots, x_v]$ with finite zero set $\cZ(\sJ) = \cS$, then $\cx[\bx]/\sJ$ is a finite-dimensional complex vector space and its dimension
is equal to the sum of the multiplicities of all the zeros of $\cI$, 
$\dim  \cx[\bx]/\sJ = \sum_{\bc \in \cS} \mult (\bc)$. The \emph{coordinate ring} 
of a variety $\cS \subseteq \cx^v$ is defined
as the quotient ring $\cx[\bx]/\cI(\cS)$ and this is naturally identified with the ring of ``polynomial functions''
on the set $\cS$. If $\sJ$ is an ideal in $\cx[\bx]$ with finite zero set $\cZ(\sJ)= \cS$, then it is well-known (e.g., 
Section 5.3, Proposition 7 in \cite{iva}) that
\begin{equation}
\label{eqn:countmults}
 \sum_{\bc \in \cS} \mult( \bc) = \dim \cx[\bx]/\sJ \ge  \dim \cx[\bx]/\Rad( \sJ) = |\cS| 
\end{equation}
where $\mult (\bc)$ is the multiplicity of point $\bc$ as a zero of $\sJ$.  This proves

\begin{prop}
\label{prop:coordring} 
With notation as above:
\begin{itemize}
\item[(i)]
If $\cS \subseteq \cZ( \sJ)$, then $\dim  \cx[\bx]/\sJ  \ge |\cS|$;
\item[(ii)]  If $\sJ$ is a radical ideal and $\cZ(\sJ)=\cS$ is finite, then $\sJ = \cI(\cS)$;
\item[(iii)]
If  $\sJ$ is an ideal in $\cx[\bx]$ with finite zero set $\cZ(\sJ) \supseteq \cS$ and the coordinate ring
$\cx[\bx]/\sJ $ has dimension equal to $|\cS|$, then the ideal $\sJ$ is radical, $\cZ(\sJ) = \cS$, each point of $\cS$ is a smooth
point (multiplicity one), and $\cI(\cS)= \sJ$.  $\Box$
\end{itemize}
\end{prop}

Let $\cG=\{f_1,\ldots, f_\ell\}$ be a generating set for ideal $\sJ \subseteq \cx[x_1,\ldots,x_v]$.  The 
\emph{Jacobian} of the system $\{f_1(\bx)=0, \ldots, f_\ell(\bx)=0\}$ of polynomial equations evaluated at 
point $\bc \in \cx^v$ is the $v\times \ell$ matrix $\Jac(\cG, \bc)$ with $(i,j)$-entry equal to
$\left. \frac{ \partial f_j }{\partial x_i } \right|_{\bc}$,   the partial derivative $\partial f_j/\partial x_i$ evaluated 
at point $\bc$. Since we are dealing only with  zero-dimensional varieties in this paper, we say the point 
$\bc$ is \emph{smooth} if $\Jac(\cG, \bc)$ has column rank $v$, so that the Zariski tangent space is 
zero-dimensional. A smooth point has multiplicity one. So another way to  show that the ideal $\sJ$ is 
radical is to check that, at each point $\bc$ of its zero set, the Jacobian 
$\Jac(\cG, \bc)$ has full row rank where $\cG$ is some generating set for $\sJ$.  

\begin{prop}
\label{prop:allradical}
Let $\cG \subseteq \cx[\bx]$ be any set of polynomials such that $\cG_0 \subseteq \cG$ (cf. Equation (\ref{eqn:Gtrivial})). Then $\langle \cG \rangle$ is a radical ideal.
\end{prop}

\begin{proof}
In the proof of Lemma \ref{lem:idealofk-sets}, we proved that the Jacobian $\Jac(\cG_0, \bc_B)$ of $\cG_0$
has column rank equal to $v$ at any characteristic vector $\bc_B$ of any $k$-set $B$. Since 
 $\Jac(\cG_0, \bc_B) $ is a submatrix of  $ \Jac(\cG, \bc_B)$, this latter Jacobian also has full row rank and 
 each point of the finite variety $\cZ( \cG)$ is smooth. 
\end{proof}

We note here that this approach also provides another proof of the fundamental bound of 
Ray-Chaudhuri and  Wilson.

\begin{lem}
\label{lem:monposet}
Let $S  \subseteq \{1,\ldots,v\}$ and $t\ge |S|$. If $\sI$ is an ideal containing $\cT$, then in the
quotient ring $\cx[\bx]/\sI$,  the coset $x^S + \sI$ is expressible
as a linear combination of cosets $x^T + \sI$ where $|T|=t$.
\end{lem}

\begin{proof}
Assume $t=|S|+1$. Since $ 1 + \sI  =  \frac{1}{k} \sum_j x_j + \sI$ we have 
$$ x^S+\sI = \prod_{s\in S} x_s + \sI =  \frac{1}{k} \sum_j x_j   \prod_{s\in S} x_s  + \sI = \left( \frac{t-1}{k} x^S + \sI 
\right) +   \left( \sum_{ \stackrel{S\subseteq T  }{ |T|=t} } x^T \right) + \sI $$
and we may solve for $x^S + \sI$.
\end{proof}

\begin{thm}
\label{thm:raychwilson}
If $(X,\cB)$ is a $2s$-$(v,k,\lambda)$ design, then $|\cB| \ge \binom{v}{s}$.
\end{thm}

\begin{proof}
Since the coordinate ring $\cx[\bx]/\cI(\cB) $ has dimension $|\cB|$
and the monomials $\{ x^C  : |C| = s \}$ represent linearly independent cosets in this quotient ring, we have
 $|\cB| \ge \binom{v}{s}$. 
\end{proof}

This language differs from that employed by Ray-Chaudhuri and Wilson. Consider the $\binom{v}{s} \times |\cB|$ matrix $A^{(s)}$ with rows indexed by all $C \subseteq X$ with $|C|=s$, with columns indexed by the set 
$\cB$ of blocks  of a $t$-($v,k,\lambda)$ design, and $(C,B)$-entry equal to one if $C\subseteq B$ and equal
to zero otherwise. The proof of Theorem 1 in \cite{raychw} establishes that the columns 
$\bc_B$ of $A^{(s)}$ span the space $\re^{\binom{v}{s}}$. The celebrated bound follows immediately for even $t$ and, for odd $t$, one obtains $|\cB| \ge 2 \binom{v}{s}$ whenever $\cB$ is the block set of a $t$-($v,k,\lambda$) design with $t=2s+1$ by applying this idea to both the derived design and the residual design. 
As we will use the fact later, we record here a 

\begin{lem}
\label{lem:RCW}
Let $(X,\cB)$ be a $t$-($v,k,\lambda$) design with $t\ge 2s$ and consider the incidence matrix $A^{(s)}$ of $s$-subsets of $X$ versus blocks as defined in the previous paragraph. Then the column rank of $A^{(s)}$ is exactly $\binom{v}{s}$.  $\Box$
\end{lem} 

%%%%%%%%%%%%%%%%%%%%%%%%%%%%%%%%%%%%%%%%%%
\section{Steiner systems and partial designs}
\label{sec:Steiner}

For a subset $C \subseteq X$ and $f \in \cR$, define $f(C)$ in the obvious way,
by setting $x_i = 1$ if $i \in C$ and $x_i = 0$ if $i \not\in C$
($1 \leq i \leq v$), and
then evaluating $f(\chi_C )$ where $\chi_C = (x_1, \dots , x_v)$.

We want to construct small-degree generating sets $\cG$ for the ideal $\sI = \cI(\cB)$ where $\cB$ is the block
set of our design $(X,\cB)$. We assume throughout that $\cG$ contains $\cG_0$ so that 
$\langle \cG\rangle$ contains the trivial ideal $\cT$. Every zero of this latter ideal is a $01$-vector with 
exactly $k$ ones. As we choose the remaining generators, we need only search for multilinear polynomials: 
each monomial $x_1^{e_1} \cdots x_v^{e_v}$ appearing in $f(\bx)$ has all $e_i \in \{0,1\}$.  It is clear that 
the automorphism group of a design $(X,\cB)$ acts on the ideal $\cI(\cB)$ by permuting indeterminates; rather than
minimizing $|\cG|$, we typically show a preference for sets of generators invariant under this action.

Recall, for a subset $C \subseteq X$, we have $x^C = \prod_{i \in C} x_i$.
Next define, for an integer $j \leq |C|$, the polynomial
\[ x^{C,j} = \sum_{J \subseteq C, |J| = j} x^J.\]
(This is a certain symmetric function --- the elementary symmetric polynomial of degree $j$ --- in the variables $\{x_i\mid i\in C\}$.)
For example, $ x^{C,|C|}=x^C$. Since the trivial ideal contains $\left( x^{X,1} - k\right)^j$ for all $j$ it also 
contains $x^{X,j} - \binom{k}{j}$ for $1\le j\le k$. For example, modulo $\langle \cG_0 \rangle$, 
\begin{eqnarray*} 
\left( x^{X,1} - k\right)^2 &=& \sum_{i=1}^v x_i^2 + 2 \left( \sum_{i<j} x_i x_j \right) -2k  \left( \sum_{i=1}^v x_i \right)
+ k^2 \\
 &\equiv&  \left(  \sum_{i=1}^v x_i - k  \right) + 2 \left( \sum_{i<j} x_i x_j - \binom{k}{2} \right) -2k  \left( \sum_{i=1}^v x_i - k \right) \\
 &\equiv&  \left(  x^{X,1} - k \right) + 2 \left(  x^{X,2} -   \binom{k}{2} \right) - 2k  \left(  x^{X,1} - k \right) .
\end{eqnarray*}

\begin{lem}
\label{lem:special}
Suppose that $B \subseteq X$, $|B| = k$  and 
$J \subseteq B$. Denote $j = |J|$. Define
\begin{equation}
\label{eqn:gbJ}
g_{B,J}(\bx) = x^{B,j}  - \binom{k}{j} x^J ~. 
\end{equation}
Then, for every $C \subseteq X$, we have that $g_{B,J}(C) \neq 0$
if and only if $j \leq |C \cap B| < k$.
\end{lem}

\begin{proof}
If $|C \cap B| < j$, then every monomial appearing in $g_{B,J}(\bx)$
evaluates to $0$ on $\chi_C$. If $|C \cap B| = k$, then all the monomials
in $g_{B,J}(\bx)$ take nonzero values on $\chi_C$, and 
$g_{B,J}(C) = \binom{k}{j} (1) - (1) \binom{k}{j} = 0$.
If $j \leq |C \cap B| < k$, then some proper subset of 
the monomials occurring in $g_{B,J}(\bx)$ take nonzero value, and $g_{B,J}(C)$ cannot
equal $0$.
\end{proof}

\begin{thm} 
\label{thm:k-unif}
For any $k$-uniform hypergraph $(X,\cB)$,
$\gamma_2( \cB)  \leq k$.
\end{thm}

\begin{proof}
In addition to our generators $\cG_0$ for $\cT$ given in (\ref{eqn:Gtrivial}), 
we will include one generator $g_Y(\bx)$ for each
set $Y$ of $k-1$ points. For $Y \subseteq X$ with $|Y|=k-1$, define
\[ J = \{ j \in X : Y \cup \{j\} \in \cB \} .\]
Suppose $J = \{j_1, \dots , j_{\ell} \}$ (possibly the empty set).
Then consider the polynomial 
\begin{equation}
\label{eqn:g_Y}
g_Y(\bx) = x^Y (x^{J,1} - 1),
\end{equation}
noting that $g_Y(\bx)=-x^Y$ whenever $Y$ is contained in no block of the hypergraph.

Let $\sI$ be the ideal generated by 
$$ \cG = \cG_0 \cup \left\{ g_Y(\bx) : |Y| = k-1 \right\}. $$
It follows from Proposition \ref{prop:allradical} that $\sI$  
is a radical ideal. In order to prove 
that  $\sI = \cI(\cB)$ (and hence that $\gamma_2(\cB)\le k$),  we must verify
\begin{itemize}
\item $f(B)=0$ for every $f\in \cG$ and every $B\in \cB$;
\item for any $k$-set $C \not\in \cB$, there is some $Y$ with $g_Y(C)\neq 0$.
\end{itemize}
Suppose that $B \in \cB$; then it is easy to verify from (\ref{eqn:g_Y}) 
that $g_Y(B) = 0$ for any $(k-1)$-subset $Y$. Now suppose that $C \subseteq X$, $|C| = k$,
$C \not\in \cB$.
Let $Y \subseteq C$, $|Y| = k-1$. If $Y$ is contained in no blocks,
then $g_Y(C) = -1$  since $g_Y(\bx)=-x^Y$. On the other hand, if $Y$ is contained in 
at least one block, then $g_Y(C) = -1$ from (\ref{eqn:g_Y}). We then have that $\sI$ is a radical
ideal with $\cZ(\sI)= \cB$; by Proposition  \ref{prop:coordring}{\it (ii)} we are done.
\end{proof}

A \emph{Steiner system} is a $t$-($v,k,\lambda$) design with $\lambda=1$. The question of existence of non-trivial Steiner systems with $t>5$ has been recently resolved in spectacular fashion by Keevash \cite{keevash}.

\begin{thm}
\label{thm:Steiner}
Let $(X,\cB)$ be any $t$-$(v,k,1)$ design. For a block $B \in \cB$ and any $t$-element subset $T$ contained
in $B$, define (as in (\ref{eqn:gbJ}))
$$  g_{B,T}(\bx) = x^{B,t}  - \binom{k}{t} x^T . $$
Then
\begin{description}
\item[(i)]  $\cI(\cB)$ is generated by $\cG_0 \cup   \left\{ g_{B,T}(\bx) : B \in \cB, T\subseteq B, |T|=t \right\}$;
\item[(ii)] $\gamma_2(\cB) \le t$.
\end{description}
\end{thm}

\begin{proof}
In view of Theorem \ref{thm:k-unif}, we can assume that
$t < k$. 
Consider the generating set 
$$ \cG = \cG_0 \cup \left\{ g_{B,T}(\bx) : B \in \cB, T\subseteq B, |T|=t \right\}. $$

From Lemma \ref{lem:special}, 
if $B \in \cB$ then $g_{B,T}(B) = 0$ for each $t$-subset $T$ of $B$, while 
if $B' \in \cB$, $B' \neq B$, then $|B' \cap B| \leq t-1$ and
it follows that $g_{B',T}(B) = 0$ for any $t$-subset $T$ of $B'$.

Now suppose that $|C| = k$, $C \not\in \cB$, and 
choose $T \subseteq C$, $|T| = t$.
There is a block $B \in \cB$ with $T \subseteq B$.
Then $g_{B,T}(C) \neq 0$ from Lemma \ref{lem:special} 
because $t \leq |B \cap C| \leq k-1$.

So $\cZ(\cG)=\{ \bc_B \mid B\in \cB\}$. By Proposition \ref{prop:allradical}, $\langle \cG\rangle$
is a radical ideal. So Equation (\ref{eqn:0satz}) gives $\cI(\cB)= \langle \cG\rangle$ and the result follows.
\end{proof}

\begin{remark}
Let $B$ be a block of the $t$-$(v,k,1)$ design $(X,\cB)$ and let $T$ and $T'$ be two $t$-element subsets of
$B$. Then it is easy to see that $\cI(\cB)$ contains $x^T - x^{T'}$. Since each of the generators $g_{B,T}$ in
the above theorem is expressible as a sum of polynomials of this form, we have a perhaps simpler set of 
generators $\cG_0 \cup \left\{ x^T - x^{T'} \mid T,T' \subseteq B \in \cB, \ |T|=|T'|=t \right\}$ for the ideal.
\end{remark}

\noindent {\bf Question:} Do the cosets $\left\{ x^{B,t} + \cI(\cB)  \mid B \in \cB \right\}$ form 
a basis for the coordinate ring $\cx[\bx]/\cI(\cB)$ in this case?

\bigskip

Next, we describe a couple of variations of Theorem
\ref{thm:Steiner}. A \emph{partial} $t$-$(v,k,1)$-design
is a $k$-uniform hypergraph in which any $t$-subset occurs
in {\it at most} one block. A partial  $t$-$(v,k,1)$-design,
$(X,\cB)$, 
is {\it maximal} if there does not exist
a $k$-subset $C \subseteq X$, $C \not\in \cB$ such that
$(X,\cB \cup \{C\})$ is a partial $t$-$(v,k,1)$-design.

\begin{cor}
\label{cor:maximal}
For any maximal partial $t$-$(v,k,1)$-design $(X,\cB)$,
$\gamma_2(\cB) \leq t$.
\end{cor}

\begin{proof}
As we employ the same generating set as in the proof of
Theorem \ref{thm:Steiner}, we need only check that $\cZ(\langle \cG \rangle) = \cB$. As before, 
we have that $g_{B,T}(B') = 0$ for all blocks $B, B' \in \cB$ and all $t$-subsets $T$ of $B$.
Now suppose that $|C| = k$, $C \not\in \cB$. Because  $(X,\cB)$ is a
maximal partial $t$-$(v,k,1)$-design, it is possible to choose $T \subseteq C$, 
$|T| = t$ such that there is a block $B \in \cB$ with $T \subseteq B$.
Then $g_{B,T}(C) \neq 0$ as before.
\end{proof}

By a slight extension of our construction, we do not
require the partial $t$-$(v,k,1)$-design to be maximal.

\begin{thm}
\label{thm:partial}
For any partial $t$-$(v,k,1)$-design $(X,\cB)$,
$\gamma_2(\cB) \leq t$.
\end{thm}

\begin{proof} If $(X,\cB)$ is maximal, 
then Corollary
\ref{cor:maximal} yields the desired result,
so assume $(X,\cB)$ is not maximal. Let 
$$ \mathbf{T} = \left\{ T \subseteq X \ : \  |T|=t, (\forall B\in \cB)( T \not\subseteq B) \right\}. $$
Now consider the generating set 
$$ \cG = \cG_0 \cup \left\{ g_{B,T}(\bx) : B \in \cB, T\subseteq B, |T|=t \right\} \cup \{ x^T : T \in \mathbf{T} \}. $$
Since $x^T$ evaluates to zero on every block, we still have $\cB \subseteq \cZ( \langle \cG\rangle )$.  But if 
$C$ is a $k$-set not belonging to $\cB$, either $\cB \cup \{C\}$ is again a partial $t$-design or some $t$-subset $T$
of $C$ is contained in some block $B$. In the latter case, $g_{B,T}(C) \neq 0$ as above; in the former case, every $t$-subset of $C$ belongs to $\mathbf{T}$ so we can
take any $t$-subset $T \subseteq C$ and, with $f(\bx)=x^T \in \cG$, we have $f(C) \neq 0$. So $\cZ( \langle \cG \rangle) = \cB$ and the rest of the proof follows just as before. 
\end{proof}

This result gives us another upper bound on $\gamma_2$ for general $t$-designs.

\begin{cor} 
\label{cor:maxintersection}
Let $(X,\cB)$ be a $t$-$(v,k,\lambda)$ design with $|B \cap B'| < s$ for every pair $B,B'$ of distinct
blocks. Then $\gamma_2(\cB) \le s$. $\Box$
\end{cor}

%%%%%%%%%%%%%%%%%%%%%%%%%%%%%%%%%%%%%%%%%%
\section{Symmetric balanced incomplete block designs}
\label{sec:symmetric}

A $2$-($v,k,\lambda$) design is traditionally called a \emph{balanced incomplete block design} (BIBD)  \cite[Chapter 1]{dougbook}. Fisher's inequality states that, for any such $2$-design, we have $|\cB| \ge |X|$ since, for $t\ge 2$,
the point-block incidence matrix $A$ has rank $v$. A $2$-design with
equally many blocks and points is called a \emph{symmetric 2-design}.  While
Theorem \ref{thm:tover2} implies here that $\gamma_1(\cB) > 1$, we may use the invertibility 
of $A$ to obtain more information in this case.  If 
$f(\bx)= w_0 + \sum_{i=1}^v w_i x_i \in \cI(\cB)$ then $\bw =(w_1,\ldots,w_v)$ satisfies $\bw^\top A = -w_0 \mathbf{1}$
and $\bw + \frac{w_0}{k} \mathbf{1}$ lies in the left nullspace of $A$. So $\bw$ is a scalar multiple of $\mathbf{1}$
and $f$ is trivial. The quadratic case is more interesting.  Let $\br_i$ denote row $i$ of matrix $A$. For any 
two distinct  points $i,j\in X$, the entrywise product $\br_i \circ \br_j$ is expressible as a linear combination of 
the rows  of $A$. Say
$$ \br_i \circ \br_j = \sum_{h=1}^v w_h \br_h ~ .$$
Then the polynomial $f(\bx)$ given by
$$ f(\bx) = x_i x_j - \sum_{h=1}^v w_h x_h $$
is easily seen to belong to $\cI(\cB)$: for a block $B$ indexing column $\ell$ of matrix $A$, we have
$f( \bc_B) =  A_{i \ell } A_{j\ell} - \sum_{h=1}^v w_h A_{h\ell } = 0$. In fact, we may determine
the coefficients $w_h$ explicitly to obtain  a nice generating set for our ideal.

\begin{thm}
\label{thm:symbibd}
Let $(X,\cB)$ be any non-trivial symmetric $2$-$(v,k,\lambda)$ design. For each pair $i,j$ of distinct points 
from $X$ define 
$$f_{i,j}(\bx) = (k- \lambda )x_ix_j - \sum_{i,j \in B\in \cB} x^{B,1} + \lambda^2 . $$
Then
\begin{description}
\item[(i)]  $\cI(\cB)$ is generated by $\cG_0 \cup \{ f_{i,j} \mid  i,j \in X \} $;
\item[(ii)] $\gamma_1(\cB) = \gamma_2(\cB)=2$;
\item[(iii)] the coordinate ring $\cx[\bx]/\cI(\cB)$ admits a basis consisting of cosets $\{ x_i + \cI(\cB) \mid 1\le i\le v\}$.
\end{description}
\end{thm}

\begin{proof} Assume $(X,\cB)$ is a $2$-design with $|\cB|=v$ and incidence matrix $A$.
As $A A^\top  = (k-\lambda)I+ \lambda J$, we see that the inverse of our incidence matrix is
$$ A^{-1} = \frac{1}{k-\lambda} \left( A^\top - \frac{\lambda}{k} J \right). $$
Letting $\br_i$ denote the $i^{\rm th}$ row of $A$ ($i\in X$), observe that $\br_i \circ \br_j$ is a 
01-vector of length $v$ with $\lambda$ entries equal to one. So $\br_i \circ \br_j = \bw^\top A$  gives
$$ \bw^\top =  \left( \br_i \circ \br_j \right) A^{-1} = \frac{1}{k-\lambda} \sum_{i,j\in B} \bc_B^\top - \frac{ \lambda^2 }{k(k-\lambda)} \mathbf{1}^\top $$
with entries 
$$ w_h = \frac{ - \lambda^2 }{k(k-\lambda)} + \frac{1}{k-\lambda} \left| \{ B \in \cB \mid h,i,j \in B \} \right|; $$
that is, $\cI(\cB)$ contains the quadratic polynomial 
$$x_i x_j +  \frac{1}{k-\lambda}  \sum_{i,j \in B} x^{B,1} 
- \frac{ \lambda^2 }{k(k-\lambda)} x^{X,1}.$$
As $x^{X,1}$ takes value $k$ on each $\bc_B$, this shows that
ideal $\sI$ contains $f_{i,j}(\bx)$ for every pair $i,j$ of distinct points.

Set $\cG = \cG_0 \cup \{ f_{i,j} \mid  i,j \in X \} $ and consider $\sI = \langle \cG\rangle$. 
The dimension of the quotient ring $\cx[\bx]/\sI$
is $v$ since every monomial of degree two is congruent, modulo $\langle \cG \rangle$, to some polynomial
of degree one\footnote{We choose some monomial ordering for the ring $\cx[\bx]$ which refines the partial order by total degree.}. Since we have $\dim \cx[\bx]/\sI = |\cB|$, we use Proposition \ref{prop:coordring}{\it (iii)} to see that 
$\cZ(\sI)=\cB$ and each zero has multiplicity one.
It then follows that $\sI  = \cI(\cB)$ and the cosets of the
form $x_ix_j + \cI(\cB)$ ($i\neq j$) form a basis as claimed.
\end{proof}

\begin{example}
Consider the case where $(X,\cB)$ is a symmetric $2$-($v,k,2$) design.
Let $X = \{i: 1 \leq i \leq v\}$ be the set of points and let 
$\cB = \{B_r: 1 \leq r \leq v\}$  be the set of blocks in the design.
Let $i,j$ be any two distinct points. There are two blocks that contain $i$ and $j$, 
say $B_r$ and $B_s$. Note that $B_r \cap B_s = \{i,j\}$.
Denote the symmetric difference by
\begin{equation*}
\label{eqn:Zij}
 Z_{i,j} = B_r \cup B_s \setminus \{i,j\} 
\end{equation*}
and define 
$$f_{i,j}(\bx) = (k-2)x_ix_j +4 -2(x_i + x_j) - \sum_{ h \in Z_{i,j} } x_h  .$$
Then $\cI(\cB)$ is generated by $x^{X,1}-k$, $x_i(x_i-1)$ ($1\le i\le v$) and the polynomials $f_{i,j}(\bx)$.
\end{example}

\begin{thm} 
\label{thm:projective}
Suppose that $(X,\cB)$ consists of the points and
$e$-dimensional subspaces of $\mathrm{PG}(d,q)$, where
$1 \leq e < d$. Let $\cL$ denote the set of all lines (1-dimensional subspaces) of $\mathrm{PG}(d,q)$. 
For every line $L$ in $\cL$ and every $2$-element subset $J
\subseteq L$  define $g_{L,J}(\bx) = x^{L,2}  - \binom{q+1}{2} x^J$.  Then
\begin{description}
\item[(i)]  $\cI(\cB)$ is generated by $\cG :=\cG_0 \cup \{ g_{L,J} \mid  L \in \cL, J \subseteq L, |J|=2  \} $;
\item[(ii)] $\gamma_1(\cB) = \gamma_2(\cB)=2$;
\item[(iii)] the coordinate ring $\cx[\bx]/\cI(\cB)$ admits a basis consisting of cosets $\{ x_i + \cI(\cB) \mid 1\le i\le v\}$.
\end{description}
\end{thm}

\begin{proof}
Here we have $v = (q^{d+1}-1) / (q-1)$ and 
$k = (q^{e+1}-1) / (q-1)$. Every line of $\mathrm{PG}(d,q)$
contains $q+1$ points.
Suppose $B$ is an $e$-dimensional subspace of 
$\mathrm{PG}(d,q)$ and $L$ is a line. Then
$|L \cap B| \in \{ 0,1,q+1\}$. In each case,
$g_{L,J}(B) = 0$ for any $2$-element $J \subseteq L$  by Lemma \ref{lem:special}.

A subset of points of $\mathrm{PG}(d,q)$
that intersects any line in $0,1$ or $q+1$ points
is necessarily a subspace of $\mathrm{PG}(d,q)$.
Suppose that $|C| = k$ but $C$ is not an
$e$-dimensional subspace of 
$\mathrm{PG}(d,q)$. Then there exists a line 
$L$ such that
$|L \cap C| \not\in 0,1,q+1$. In this case,
$g_{L,J}(C) \neq 0$  by Lemma \ref{lem:special}.

We have shown that the ideal generated by  $\cG$ 
is a radical ideal with zero set $\cB$, so we are done by Proposition \ref{prop:coordring}{\it (ii)}.
\end{proof}

\begin{remark}
Clearly we can build a much smaller generating set than our choice of $\cG$ by selecting just one pair $J$ of 
points in each line. We instead prefer here to choose a set $\cG$ of polynomials which is invariant under the automorphism group of the design.
\end{remark}

%%%%%%%%%%%%%%%%%%%%%%%%%%%%%%%%%%%%%%%%%%
\section{Triple systems}
\label{sec:triples}

Identifying each square-free monomial $x^C$ with the set $C$, the multilinear polynomials with real coefficients
are in bijective correspondence with the real-valued functions on the Boolean lattice. For multilinear $f$ write
$$ c(f) = \left( f_D : D \subseteq \{1,\ldots,v\} \right) \qquad \mbox{where} \qquad f(\bx)= \sum_D f_D x^D .$$

Suppose that $C \subseteq X$ and $0 \leq s \leq |C|$. 
The {\it $s$-incidence vector} of $C$, denoted
$\delta = \delta^s(C)$, is the vector of length 
$w = \sum_{i=0}^{s} \binom{v}{i}$, whose coordinates
correspond to the subsets of $X$ of cardinality at most $s$, defined by
\[ \delta_J =
\begin{cases} 1 & \text{if $J \subseteq C$} \\
0 & \text{otherwise,}
\end{cases}
\]
where $|J| \leq s$.

Now suppose that $f$ is a multilinear polynomial in $x_1, \dots , x_v$
of degree at most $s$.
The vector of coefficients of $f$, denoted $c = c(f)$,
is also a $w$-dimensional vector whose coordinates
correspond to the subsets of $X$ of cardinality at most $s$.

The following lemma is obvious.
\begin{lem} 
\label{lem:innerprod}
If $f$ is a polynomial of degree
at most $s$ and $C \subseteq X$, then $f(C) = \delta \cdot c$, where
$\delta = \delta^s(C)$ and $c = c(f)$.
\end{lem}

\begin{example}
Suppose that $X = \{1,2,3,4,5\}$, $C = \{1,2,3\}$ and $s=2$.
Let \[f = 1 + x_1 + 2x_2 - 3x_3 + 4x_4 - x_1x_2 + +3x_1x_5 + 2x_2x_3 - 3x_3x_5.\]
Then        \[c(f)= (1,1,2,-3,4,0,-1,0,0,3,2,0,0,0,-3,0)\]
and \[\delta^s(C) = (1,1,1, 1,0,0, 1,1,0,0,1,0,0,0, 0,0).\] 
It is easy to verify that $f(C) = 1+1+2-3-1+2 = 2$.
\end{example}

\begin{thm}
\label{thm:lowerbound}
Suppose that $(X, \cB)$ is a $2$-$(v,3,2)$ design
such that \[ \{ \{1,2,3\}, \{1,4,5\}, \{2,4,6\}, \{3,5,6\},
\{1,2,4\}, \{1,3,5\}, \{2,3,6\} \} \subseteq \cB,\]
but 
\[ \{4,5,6\} \not\in \cB.\]
Then $\gamma_2(\cB) =3$.
\end{thm}

\begin{proof}
By Theorems \ref{thm:tover2} and \ref{thm:k-unif}, we have $2 \le \gamma_2(\cB) \le 3$. We now show 
$ \gamma_2(\cB)  > 2$.
Denote $B_1 = \{1,2,3\}$, $B_2 = \{1,4,5\}$, $B_3 = \{2,4,6\}$, 
$B_4 = \{3,5,6\}$, $B_5 = \{1,2,4\}$, $B_6 = \{1,3,5\}$, 
$B_7 = \{2,3,6\}$ and $C = \{4,5,6\}$.
For $s=2$, it is easy to verify that
\begin{equation}
\label{eq5} \delta^s(B_1) + \delta^s(B_2)  + \delta^s(B_3)  
+ \delta^s(B_4) 
=   \delta^s(B_5)  + \delta^s(B_6)  + \delta^s(B_7)  + \delta^s(C) .\end{equation}
For any  $f \in \cI(\cB)$, we have that
$f(B_1) = f(B_2) = \cdots = f(B_7) = 0$.
It then follows from (\ref{eq5}) and Lemma \ref{lem:innerprod} that
$f(C) = 0$. However, since $C \not\in \cB$, it must be the
case that $f(C) \neq 0$ for some $f \in  \cI(\cB)$. This contradiction
establishes the desired result.
\end{proof}

Note that this linearization technique may be applied more generally. If we can find $C \not\in \cB$ such that
$\delta^s(C)$ is a linear combination of the $w$-dimensional vectors $\{ \delta^s(B) \mid  B\in \cB\}$,
then $\gamma_2(\cB) > s$.

We now show one way to construct examples of $2$-$(v,3,2)$ designs that satisfy the hypotheses
of Theorem \ref{thm:lowerbound}. Our construction works for any $v \equiv 1,3 \pmod{6}$,
$v \geq 15$.

Suppose we have a $2$-$(v,3,2)$ design that satisfies the hypotheses
of Theorem \ref{thm:lowerbound}.
We first observe that
\begin{equation}
\label{set1} \{ \{1,2,3\}, \{1,4,5\}, \{2,4,6\}, \{3,5,6\} \}
\end{equation}
is a set of four blocks that forms a so-called \emph{quadrilateral}
(or \emph{Pasch configuration}). 
As well, 
\begin{equation}
\label{set2}
\{ \{1,2,4\}, \{1,3,5\}, \{2,3,6\} \} 
\end{equation}
is a set of three blocks that is not contained
in a quadrilateral (because $\{4,5,6\}$ is not a block).

We require two ingredients:
\begin{enumerate}
\item
The unique $2$-$(7,3,1)$ design is isomorphic to point-line structure of PG$(2,2)$ and it contains a quadrilateral
(in fact, it contains exactly seven distinct quadrilaterals). 
Therefore the  points of this design 
can be relabelled so it contains the four blocks in (\ref{set1}).
Now, from the Doyen-Wilson Theorem, we can embed this 
$2$-$(7,3,1)$ design  in a $2$-$(v,3,1)$ design for any $v \equiv 1,3 \pmod{6}$,
$v \geq 15$.
\item 
It is shown in \cite[Theorem 3.1]{SW} that the maximum number of quadrilaterals in
a $2$-$(v,3,1)$ design is $v(v-1)(v-3)/24$, and this maximum is attained if an only if
the design is isomorphic to the point-line structure of the projective geometry PG$(n,2)$ for some integer $n \geq 2$.
Take any $2$-$(v,3,1)$ design that is not isomorphic to the projective geometry
PG$(n,2)$ (this can be done provided $v \equiv 1,3 \pmod{6}$,
$v \geq 9$. It is easy to see that design must contain three non-collinear points that are not contained
in a quadrilateral. By relabelling points in the design, we can assume that 
the three non-collinear points 
are denoted $1,2$ and $3$, and they are contained in the three blocks
in (\ref{set2}). Moreover, $\{4,5,6\}$ is not a block in this design because the three points
$1,2,3$ are not contained in a quadrilateral.
\end{enumerate}
Now we take the union of the blocks in the two 
$2$-$(v,3,1)$ designs constructed above. The result is a $2$-$(v,3,2)$ design that contains
the seven blocks in  (\ref{set1}) and  (\ref{set2}). We have already noted
that $\{4,5,6\}$ is not a block in the second design. It is also not a block in the
first design because the pairs $\{4,5\}$, $\{4,6\}$ and $\{5,6\}$ occur in three
different blocks in this design.

As a consequence of this discussion and Theorem \ref{thm:lowerbound}, the following
result is immediate.

\begin{thm}
Suppose $v \equiv 1,3 \pmod{6}$, $v \geq 15$. Then there exists a 
$2$-$(v,3,2)$ design  $(X,\cB)$ such that  $\gamma_2(\cB)=3$.
\end{thm}

%%%%%%%%%%%%%%%%%%%%%%%%%%%%%%%%

We now give a more general version of this construction, starting from an
arbitrary trade. We recall some definitions from 
\cite{FGG}. A \emph{trade} is a set $\mathbf{T}$ of two (finite) subsets of blocks of
size three, say $\mathbf{T} = \{T_1, T_2\}$ that satisfies the following
properties:
\begin{enumerate}
\item each $T_\ell$ ($\ell = 1,2$) is a partial Steiner triple system (i.e., no pair of points
occurs in more than one block)
\item $T_1 \cap T_2 = \emptyset$
\item the set of pairs contained in the blocks in $T_1$ is identical to the set of pairs
contained in the blocks in $T_2$.
\end{enumerate}

As an example, if 
\[T_1 = \{ \{1,2,3\}, \{1,4,5\}, \{2,4,6\}, \{3,5,6\} \} \]
and 
\[T_2 = \{ \{1,2,4\}, \{1,3,5\}, \{2,3,6\}, \{4,5,6\} \} ,\]
then $\mathbf{T} = \{T_1, T_2\}$ is a trade.

The \emph{volume} of a trade $\mathbf{T} = \{T_1, T_2\}$, which is denoted
$\mathit{vol}(\mathbf{T})$, is the number of blocks in $T_1$ 
(or, equivalently, the number of blocks in $T_2$).
The \emph{foundation} of $\mathbf{T}$, denoted $\mathit{found}(\mathbf{T})$, is the set of points
covered by the blocks in $T_1$ (or, equivalently, the set of points
covered by the blocks in $T_2$).

In the above example, $\mathit{vol}(\mathbf{T}) = 4$ and 
$\mathit{found}(\mathbf{T}) = \{1,2,3,4,5,6\}$.

The following lemma is easy to prove. 

\begin{lem}
\label{lem:trade}
$\mathbf{T} = \{T_1, T_2\}$ is a trade. Then
\[ \sum_{B \in T_1} \delta^2(B) = \sum_{B \in T_2} \delta^2(B).\]
\end{lem}

\begin{proof}
The definition of a trade $\mathbf{T} = \{T_1, T_2\}$ 
ensures that $T_1$ and $T_2$ cover the same set of pairs.
So we just need to prove that $T_1$ and $T_2$ contain the same points with the 
same multiplicities. Suppose $i \in \mathit{found}(\mathbf{T})$. Let 
\[ N_\ell = \{ j : \{i,j\} \text{ is contained in a block in } T_1 \} .\]
Then it is easy to see that $i$ is contained in $|N_\ell|/2$ blocks in each of
$T_1$ and $T_2$.
\end{proof}

The following is a slight generalization of Theorem \ref{thm:lowerbound}.
We omit the proof, which makes use of Lemma \ref{lem:trade}, 
since it is essentially the same.

\begin{thm}
\label{thm:lowerbound2}
Let $\mathbf{T} = \{T_1,T_2\}$ be a trade. 
Suppose $B = \{h,i,j\} \in T_2$.
Suppose that $(X, \cB)$ is  $2$-$(v,3,2)$ design
such that \[ T_1 \cup (T_2 \setminus \{B\}) 
\subseteq \cB\]
and \[B \not\in \cB.\]
Then $\gamma_2(\cB) =3$.
\end{thm}

\begin{thm}
Suppose that $\mathbf{T} = \{T_1,T_2\}$ is a trade, where
$|\mathit{found}(\mathbf{T})| = n$. Let $B \in T_1$. 
Suppose $v \equiv 1,3 \pmod{6}$, $v \geq 2n+3$. Then there exists a 
$2$-$(v,3,2)$ containing all the blocks in $T_1 \cup (T_2 \setminus \{B\})$, 
such that  $\gamma_2(\cB) =3$.
\end{thm}

\begin{proof}
$T_1$ is a partial Steiner triple system on $n$ points.
The famous result of Bryant and Horsley \cite{BH} shows that $T_1$ can be embedded
in a $2$-$(v,3,1)$ design for any $v \equiv 1,3 \pmod{6}$, $v \geq 2n+1$. 

Suppose $B = \{ h,i,j \}$ and let 
$B^* = \{i,j,\ell \}$, where $\ell \not\in \mathit{found}(\mathbf{T})$.
Define $T_2^* =  (T_2 \setminus \{B\}) \cup \{B^*\}$.
$T_2^*$ is a partial Steiner triple system on $n+1$ points, so (again, from \cite{BH})
it can be embedded
in a $2$-$(v,3,1)$ design for any $v \equiv 1,3 \pmod{6}$, $v \geq 2(n+1)+1$. 
So for $v \equiv 1,3 \pmod{6}$, $v \geq 2n+3$, we have two
$2$-$(v,3,1)$ designs (which we can assume are defined on the same set of points),
say $(Y,\cB_1)$ and $(Y,\cB_2)$, such that $T_1 \subseteq \cB_1$ and
$T_2^* \subseteq \cB_2$. Then $(Y,\cB_1 \cup \cB_2)$ is a $2$-$(v,3,2)$ design that contains
all the blocks in $T_1 \cup (T_2 \setminus \{B\})$. 

We claim that $B$ is not a block in $\cB_1 \cup \cB_2$. First, 
there is a unique block  $B' \in \cB_1$ that contains the pair $\{i,j\}$, and this
block $B'$  is one of the blocks in $T_1$. Because $B' \in T_1$ and $B \in T_2$, 
it follows that $B' \neq B$. Therefore $B \not\in \cB_1$.
To see that $B \not\in \cB_2$, we observe that the unique block in $\cB_2$ that contains
the pair $\{i,j\}$ is $B^* \neq B$.

Thus we have shown that $(Y,\cB_1 \cup \cB_2)$ satisfies the hypotheses of 
Theorem \ref{thm:lowerbound2}, and the proof is complete.
\end{proof}

%%%%%%%%%%%%%%%%%%%%%%%%%%%%%%%%%%%%%%%%%%
\section{Strength greater than two} 
\label{sec:highstrength}

We finish by addressing the ideals of non-trivial $t$-($v,k,\lambda$) designs with $t>2$. For Steiner systems (where 
$\lambda=1$), we have
$$ \frac{t+1}{2} \le \gamma_1 (\cB)  \le \gamma_2 (\cB) \le t $$
using Theorem \ref{thm:tover2} and Theorem \ref{thm:Steiner}. When $\lambda > 1$, the lower
bound still holds and we may apply Lemma \ref{lem:gam1upper}: since the number of blocks is 
$$ | \cB | = \lambda \binom{v}{t} \binom{k}{t}^{-1}, $$
we find $\gamma_1(\cB) \le s$ whenever $\binom{v}{s} \binom{k}{t} > \lambda \binom{v}{t} $. We also have 
Corollary \ref{cor:maxintersection} which tells us that $\gamma_2(\cB) \le m+1$ when $m$ is 
the maximum size of the intersection of two distinct blocks.

Let $(X,\cB)$ be a $t$-($v,k,\lambda$)  design. For $i\in X$, the \emph{derived design} of $(X,\cB)$  with respect
to $i$ is the ordered pair $(\dot{X}, \dot{\cB})$ where $\dot{X} = X \setminus \{i\}$ and 
$$ \dot{ \cB } = \left\{ B\setminus \{i\} \mid i \in B \in \cB \right\}. $$
The \emph{residual design} of $(X,\cB)$  with respect to $i$ has vertex set $\dot{X}$ and block set $\{ B\in \cB \mid i \not\in B\}$.  If $(X,\cB)$ is a $t$-design, then both its derived design and its residual design are $(t-1)$-designs. 

\begin{lem}
\label{lem:derived}
Let $(X,\cB)$ be a non-trivial $t$-design with $i \in X$. With notation as above $\gamma_h(\dot{\cB}) \le \gamma_h(\cB)$ for $h=1,2$. The same inequalities hold for the residual design. 
\end{lem}

\begin{proof}
We handle the case of the derived design; the computations for the residual design are similar. To simplify
the notation, we take $i=1$.  Let $\sI = \cI(\cB)$
and define ideal $\sJ$ as the image of $\sI$ under the ring homomorphism $\varphi : \cx[x_1,\ldots,x_v] \rightarrow  \cx[x_2,\ldots,x_v] $ mapping $x_1$ to $1$ and mapping each $x_j$ to itself for $j=2,\ldots,v$.  For $g\in \sI$ write 
$\dot{g} := \varphi(g) \in \sJ$. For any $(k-1)$-set $C \subseteq \{2,\ldots, v\}$, we have $\dot{g}(C) = g( C \cup \{1\})$
and so, for $C \in \dot{\cB}$ we have $\dot{g}(C)=0$ for all $\dot{g} \in \sJ$ and, for $C \not\in \dot{\cB}$, there exists
$\dot{g} \in \sJ$ for which $\dot{g}(C) \neq 0$. It follows that, if $\cG$ is a generating set for $\sI$, then $\varphi(\cG)$
is a generating set for $\sJ$. This shows $\gamma_2(\dot{\cB}) \le \gamma_2(\cB)$. Next, if $g$ is a non-trivial
polynomial in $\sI$ of smallest degree, then $\dot{g}$ has degree no larger than the degree of $g$ and is also non-trivial since $\varphi$ maps trivial ideal to trivial ideal. 
\end{proof}

We illustrate this and other results in this paper by recording, in the following table, the exact value of
these parameters for the Witt designs and the $t$-designs appearing as their derived designs.
Up to isomorphism, there are unique block designs with parameters $5$-$(24,8,1)$, 
$4$-$(23,7,1)$,  $3$-$(22,6,1)$,  $5$-$(12,6,1)$,  $4$-$(11,5,1)$ and $3$-$(10,4,1)$. One accessible source of information on the Witt designs is the note \cite{brouwerWitt} by Andries Brouwer.

\begin{center}
\begin{tabular}{|c|c|c|l|} \hline $t$-($v,k,\lambda$) & $\gamma_1(\cB)$ & $\gamma_2(\cB)$ &  Notes \\ \hline
$5$-$(24,8,1)$ &           3         &         3      &    Theorems \ref{thm:tover2}, \ref{thm:witt24} \\
$4$-$(23,7,1)$ &           3         &         3      &    Theorem \ref{thm:witt23} \\          
$3$-$(22,6,1)$ &           2         &         2      &    Theorem \ref{thm:witt22} \\     
$2$-$(21,5,1)$ &           2         &         2      &    Theorem \ref{thm:projective} \\ \hline
$5$-$(12,6,1)$ &          3          &        3      &   discussion below \\
$4$-$(11,5,1)$ &          3          &        3      &    Lemma \ref{lem:derived} \\
$3$-$(10,4,1)$ &          2         &         2       &     discussion below \\
$2$-$(9,3,1)$  &           2         &         2       &     Theorems \ref{thm:tover2}, \ref{thm:Steiner} \\ \hline
\end{tabular}
\end{center}

In the large Witt design, with parameters $5$-($24,8,1$), blocks intersect in 0, 2 or 4 points. So we might start 
with the zonal polynomials
$$ (\bc_B \cdot \bx )  (\bc_B \cdot \bx - 2)  (\bc_B \cdot \bx - 4)  (\bc_B \cdot \bx - 8) $$
where $B \in \cB$. We know that the blocks of this design are the supports of the minimum weight codewords
in the extended binary Golay code.  We may then use the fact that this is a self-dual code to show that these,
together with the generators of $\cT$, generate our ideal.  But we can do better.

\begin{thm}
\label{thm:witt24}
Let $(X,\cB)$ be the $5$-$(24,8,1)$ design. For a block $B \in \cB$ and points $i,j\in B$, define
$$ f_{B,i,j}(\bx) = (x_i-x_j )  (\bc_B \cdot \bx - 2)  (\bc_B \cdot \bx - 4)  . $$
Then
\begin{description}
\item[(i)]  $\cI(\cB)$ is generated by $\cG_0 \cup \{ f_{B,i,j} \mid i,j \in B\in \cB  \} $;
\item[(ii)] $\gamma_1(\cB) = \gamma_2(\cB) = 3$.
\end{description}
\end{thm}

\begin{proof}
We know $\gamma_1(\cB) \ge 3$ by Theorem \ref{thm:tover2}. 

For any block $B\in \cB$, the number $m_i$  of  blocks $B' \in \cB$ with $|B \cap B'| = i$ is given by
\begin{center}
\begin{tabular}{c|rrrrrrrrr}
$i$ &   8      &    7     &     6     &    5    &    4    &    3    &    2     &    1    &     0    \\ \hline
$m_i$  &1&0&0  &   0     & 280   &   0     &   448  &   0     &   30
\end{tabular}
\end{center}
So the zonal polynomial (cf. Corollary \ref{cor:degree})  $\prod_{i=8,4,2,0} ( \bc_B \cdot \bx - i)$ belongs to 
$\cI(\cB)$ and the quadratic polynomial $(\bc_B \cdot \bx-2)(\bc_B \cdot \bx-4)$ vanishes on every block
except $B$ itself and those blocks disjoint from $B$.  But if $i$ and $j$ both belong to block $B$,  the linear 
function $x_i-x_j$ vanishes on $\bc_B$ and on $\bc_{B'}$ for any block $B'$ disjoint from $B$. 
To show that the polynomials $f_{B,i,j}$  --- as $B$ ranges over the blocks and $i$, $j$ range over the elements 
of $B$ --- together with the polynomials in the trivial ideal, generate our ideal,
we employ two basic facts about the extended binary Golay code $\mathsf{G}_{24}$. The blocks in 
$\cB$ are precisely the  supports of minimum weight codewords in this code. This is a self-dual code, 
so a binary tuple $c \in \FF_2^{24}$ satisfies  $c \in \mathsf{G}_{24}$ if and only if the mod 2
dot product $c \cdot c'$ is zero for every $c' \in \mathsf{G}_{24}$. Since $\mathsf{G}_{24}$ is generated by its 
weight eight codewords, we may say $c \in \mathsf{G}_{24}$ if and only if its inner product with these 759 
codewords is zero mod two. Since we only want to recover the codewords of weight eight, we may omit 
integer inner product six. 

Let $\sI = \langle \cG_0 \cup \{ f_{B,i,j} \mid i,j\in B\in \cB  \} \rangle$ and observe that any element of $\cZ(\sI)$
has exactly eight entries equal to one and sixteen entries equal to zero. For $c \in \FF_2^{24}$, 
let $\bc \in \re^{24}$ be the corresponding 01-vector with real entries:  $\bc_j=1$ if $c_j=1$ and $\bc_j=0$ 
if $c_j=0$.  Assuming $\bc \in \cZ(\sI)$, we have  $f_{B,i,j}(\bc)=0$ for each $B\in \cB$ and each $i,j\in B$ which implies that either $\bc_B \cdot \bc \in \{2,4\}$ or $c_i=1 \Leftrightarrow c_j=1$ for all $i,j\in B$. This latter
alternative clearly means that either $\bc=\bc_B$ or $\bc \cdot \bc_B=0$.  For the corresponding binary 
vectors, this implies $c \cdot c' =0$  for each $c' \in  \mathsf{G}_{24}$ with Hamming weight eight. 
As outlined above, this gives $c\in \mathsf{G}_{24}$
and, in turn, $\bc = \bc_{B'}$ for some $B'\in \cB$.  By Propositions 
\ref{prop:allradical} and \ref{prop:coordring}{\it (ii)}, we have $\sI=\cI(\cB)$. 
\end{proof}

\begin{thm}
\label{thm:witt23}
Let $(X,\cB)$ be the $4$-$(23,7,1)$ design. For three distinct points $i$, $j$ and $k$, let $\cC_{i,j,k} = \left\{ C=B \setminus \{i,j,k\} \mid \{i,j,k\}  \subseteq B\in \cB \right\}$ and define
$$ h_{i,j,k}(\bx) = 3 + 12 x_i x_j x_k - 3 ( x_i x_j + x_i x_k + x_j x_k ) - \sum_{C \in \cC_{i,j,k} } x^{C,2}~. $$
Then
\begin{description}
\item[(i)]  $\cI(\cB)$ is generated by $\cG_0 \cup \{ h_{i,j,k} \mid i,j,k \} $;
\item[(ii)] $\gamma_1(\cB) = \gamma_2(\cB)=3$;
\item[(iii)] the coordinate ring $\cx[\bx]/\cI(\cB)$ admits a basis consisting of those $\binom{23}{2}$ cosets
$x_i x_j + \cI(\cB)$ represented by multilinear monomials of degree two.
\end{description}
\end{thm}

\begin{proof} 
Denote by $B_1,B_2,B_3,B_4,B_5$ the five blocks containing $T:=\{i,j,k\}$ and set $C_\ell = B_\ell \setminus T$. 
We know that two distinct blocks of our Witt design intersect in either three points or one point.
We now show that $h_{i,j,k}(B)=0$ for each $B \in \cB$. To illustrate the simple arithmetic involved, we write
$$ h_{i,j,k}(\bx) = 3 + 12( x_i x_j x_k) - 3 ( x_i x_j + x_i x_k + x_j x_k ) - ( x^{C_1,2} )  - ( x^{C_2,2} )  - ( x^{C_3,2} )  - ( x^{C_4,2} )  - ( x^{C_5,2} ) $$
where $B_\ell = C_\ell \cup T$ and we retain most of the parentheses here in our evaluation.

\begin{description}
\item[case (1)]\mbox{\quad} $T \subseteq B$. Here, we have $B=  B_\ell$ for some $\ell\in \{1,\ldots,5\}$ and
$$ h_{i,j,k}(B) =  3 + 12(1) - 3(1+1+1) - (1+1+1+1+1+1)  - (0)  - (0)  - (0)  - (0) = 0 .$$ 
\item[case (2)]\mbox{\quad} $|T\cap B| = 2$. Here we must have $|B \cap B_\ell| =3$ for all $\ell$ and, as $B$
never contains two points from the same ``sub-block'' $C_\ell = B_\ell \setminus T$, we have
$$ h_{i,j,k}(B) =  3 + 12(0) - 3(1) - (0)  - (0)  - (0)  - (0)  - (0) = 0 .$$ 
\item[case (3)]\mbox{\quad} $|T\cap B| = 1$. In this case, as there are five blocks containing $T$ and $|B|=7$, we must have $|B \cap B_\ell| =3$ for exactly three values of  $\ell$ and, as $B$ contains two points from the same sub-block $C_\ell = B_\ell \setminus T$ in each of these three cases, we have
$$ h_{i,j,k}(B) =  3 + 12(0) - 3(0) - (1) - (1) - (1)  - (0)  - (0) = 0 .$$ 
\item[case (3)]\mbox{\quad} $T\cap B = \emptyset$. In this case, as there are five blocks containing $T$ and $|B|=7$, we must have $|B \cap C_\ell | =3$ for some unique sub-block  $C_\ell = B_\ell \setminus T$ and $B$ must contain 
a unique point from each of the other four. In this case, we have
$$ h_{i,j,k}(B) =  3 + 12(0) - 3(0) - (1+1+1) - (0)  - (0)  - (0)  - (0)  = 0 .$$ 
\end{description}

On the other hand, if $S$ is a 7-set of points which is not a block, then $h_{i,j,k}(S) \neq 0$ for any $\{i,j,k\}\subseteq S$. For if $T:=\{i,j,k\}$ is contained in $S$ and $h_{i,j,k}(S) = 0$, then we have
$$ 0 = h_{i,j,k}(S) = 3 + 12(1) -3(1+1+1) - \binom{m_1}{2}  - \binom{m_2}{2}  - \binom{m_3}{2}  - \binom{m_4}{2}  - \binom{m_5}{2} $$
where $m_\ell = |S \cap C_\ell|$. But $m_1+\cdots+m_5 = 4$ and we see that the only arrangement 
that achieves the stated equality is where some $m_\ell = 4$. But then $S = B_\ell$ and we are done. So, if 
$$ \sI = \langle \cG_0 \cup \{ h_{i,j,k} \mid i,j,k \} \rangle$$
we have shown $\cZ(\sI) = \{ \bc_B \mid B \in \cB \} $.  By Proposition \ref{prop:allradical}, $\sI$ is a radical ideal.
so Proposition \ref{prop:coordring}{\it (ii)} gives us $ \sI  =\cI (\cB)$.
This proves that $\gamma_2(\cB)=3$ and, by Theorem \ref{thm:tover2}, $\gamma_1(\cB)=3$ as well.

By Lemma \ref{lem:monposet}, each coset $x_i+\sI$  can be expressed as a linear combination 
of cosets $x_ix_j  + \sI$. Since the number of blocks of the design is $\binom{23}{2}$, we see that the
cosets $\{ x_i x_j + \sI \mid i,j\}$ form a vector space basis for $\cx[\bx]/\sI$. 
\end{proof}

Next, if $(X,\cB)$ denotes the Witt design on $22$ points, Lemma \ref{lem:derived} tells us
$2 \le  \gamma_1(\cB)  \le \gamma_2(\cB) \le 3$. We now show that both values are equal to two.
 
\begin{thm}
\label{thm:witt22}
Let $(X,\cB)$ be the $3$-$(22,6,1)$ design. For two distinct points $i$, $j$ and a block $B$ containing them,
say $B=\{i,j,r,s,t,u\}$, define 
$$ h_{i,j,B}(\bx) = (x_i - x_j)( x_r + x_s + x_t + x_u - 1). $$
Then
\begin{description}
\item[(i)]  $\cI(\cB)$ is generated by $\cG_0 \cup \{ h_{i,j,B} \mid i\neq j, \ i,j \in B, \ B \in \cB \} $;
\item[(ii)] $\gamma_1(\cB) = \gamma_2(\cB)=2$;
\item[(iii)] the coordinate ring $\cx[\bx]/\cI(\cB)$ admits a basis consisting of the $77$ cosets
$x^{B,2}+\cI(\cB)$ obtained as $B$ ranges over the blocks of the design.
\end{description}
\end{thm}

\begin{proof} By  Theorem  \ref{thm:tover2}, we have $\gamma_1(\cB) \ge 2$. Consider $B\in \cB$ and two
distinct points $i,j \in B$. Write $B=\{i,j,r,s,t,u\}$. Since any two blocks of this design intersect in zero or two 
points, any block  $B'$ that contains exactly one of
$i,j$ contains exactly one  element from $\{r,s,t,u\}$. So $h_{i,j,B}(B') =0$. The same holds if $|B' \cap \{i,j\}|$ is 
even.  On the other hand, let $C$ be a 6-element subset of $X$ and choose three distinct points $i,t,u \in C$.
There is a unique block $B$ containing these three points, say $B=\{i,j,r,s,t,u\}$.  If all three polynomials
$h_{i,j,B}(\bx)$, $h_{i,r,B}(\bx)$, $h_{i,s,B}(\bx)$ vanish on $C$, then we must have $C=B$. This finishes
the proof that $\cZ(  \cG ) = \{ \bc_B \mid B \in \cB \}$ and, as $\cG_0 \subseteq \cG$, the ideal 
$\langle \cG\rangle$ is radical, proving {\bf (i)} and {\bf (ii)}. 

To show that the functions on $\cB$ represented by the polynomials $\{ x^{B,2} \mid B \in \cB\}$ are 
linearly independent, consider
the $77 \times 77$ matrix $M$ with $(B,B')$-entry equal to  the value the polynomial $ x^{B',2} $ takes at the 
point $\bc_{B}$. Then $M-I$ is the adjacency matrix of a well-known\footnote{See, for example, \url{https://www.win.tue.nl/~aeb/graphs/srg/srgtab51-100.html},} strongly regular graph with eigenvalues $60$, $5$ and $-3$.
It follows that $M$ is invertible and the 77 cosets $\{  x^{B,2} + \cI(\cB) \mid B \in \cB\}$ are linearly independent in
the coordinate ring.
\end{proof}

For the small Witt designs, we do not have a computer-free proof of our claims. Let us instead describe
some generators for the ideals. 

First consider the unique Witt design on twelve points.
Let $(X,\cB)$ be the $5$-$(12,6,1)$ design. For three distinct points $i$, $j$ and $k$, let $C=X \setminus \{i,j,k\}$.
The twelve blocks containing $i$, $j$ and $k$ yield a $2$-($9,3,1$) design
$$ (C,\cB'), \qquad \cB'= \{ B \setminus \{i,j,k\} \mid i,j,k \in B \in\cB \} $$
on the point set $C$ and the four parallel classes of this affine plane may be oriented in a total of sixteen ways (each resulting in a 4-set of directed triples of blocks). We find that certain orientations yield polynomials of degree three
which, together with those polynomials in $\cG_0$, generate the ideal $\cI(\cB)$. This shows
$$\gamma_1(\cB) = \gamma_2(\cB)=3. $$

To be precise, let $\mathsf{M}_{12}$ be the subgroup of $S_{12}$ generated by 
$$ \{  (1 \ 4)(3 \ 10)(5 \ 11)(6 \ 12), \ \  (1 \ 8 \ 9)(2 \ 3 \ 4)(5 \ 12 \ 11)(6 \ 10 \ 7) \} $$
and consider the 132 6-sets in the orbit containing $\{1,2,3,4,5,9\}$. Since $\mathsf{M}_{12}$ is 5-transitive, this is a
$5$-($12,6,1$) design. Two parallel lines in the derived design consisting of all blocks containing points
$1$, $2$ and $3$ are  $\{4,5,9\}$ and $\{ 8,10,11\}$. With a computer, one easily verifies that the polynomial
\begin{eqnarray*}
F(\bx) &=& 
x_1 x_4 ( x_{10} - x_{11} ) +
 x_1 x_5 ( x_{11} - x_{8} ) +
 x_1 x_9 ( x_{8} - x_{10} ) +\\
&& x_2 x_9 ( x_{10} - x_{11} ) +
 x_2 x_4 ( x_{11} - x_{8} ) +
 x_2 x_5 ( x_{8} - x_{10} ) + \\
&& x_3 x_5 ( x_{10} - x_{11} ) +
 x_3 x_9 ( x_{11} - x_{8} ) +
 x_3 x_4 ( x_{8} - x_{10} ) 
 \end{eqnarray*}
 vanishes on each block of the design.
It follows that any image
$$    [1^\pi, 2^\pi, 3^\pi], [4^\pi, 5^\pi, 9^\pi], [10^\pi, 11^\pi, 8^\pi] $$
of the three triples of indices under any $\pi \in \mathsf{M}_{12}$ yields another polynomial in the ideal.
It requires a bit more computation, using the {\sc Singular} computer algebra system \cite{gps},  to check that 
the ideal $\cI(\cB)$ is generated by these polynomials together with those in $\cG_0$. The relative orderings within
the three triples above is important and we do not have an intrinsic description of the permissible orderings 
that yield vanishing polynomials.

If we take the above computation as correct, we may determine $\gamma_1$ and $\gamma_2$ for the 
Witt design on eleven points.  If $(X,\cB)$ now denotes a $4$-($11,5,1$) design, we may  use Theorem \ref{thm:tover2}  to see that $\gamma_1(\cB) \ge 3$. By Lemma \ref{lem:derived},
we have equality, and $\gamma_2(\cB)=3$ as well.

Without proof, we note that if $(X,\cB)$ is the unique $3$-($10,4,1$) design \cite[Fig.~9.1]{dougbook}, we 
find $\gamma_1(\cB)=\gamma_2(\cB)=2$.  In addition to the generators of the trivial ideal, we build certain
quadratic generators from any pair $B_1,B_2$ of disjoint blocks. In order to explain these generators, we first 
describe the design.

For the construction in \cite{dougbook}, we take $X = \FF_3^2 \cup \{ \infty \}$ with the numbering 
\begin{center}
\begin{tabular}{|c|ccc|ccc|ccc|} \hline 0 & 1&2&3    &    4&5&6   &   7&8&9   \\ \hline
$\infty$ &  (0,0)&(1,0)&(2,0)   &     (0,1)&(1,1)&(2,1)   &    (0,2)&(1,2)&(2,2)  \\ \hline \end{tabular} 
\end{center}
\noindent and blocks $\{0\} \cup \ell$ where $\ell$ is a line of $AG(2,3)$ and the following eighteen symmetric differences of 
orthogonal lines:
\begin{center}
1245, \ 1278, \ 1269, \ \ 1346, \ 1379, 1358, \ \ 2356, \ 2389, \ 2347, \\
4578, \ 4679, \ 5689, \ \ 1567, \ 2468, \ 3459, \ \ 1489, \ 2579, \ 3678.
\end{center}
For any pair $B_1,B_2\in \cB$ with 
$B_1 \cap B_2 = \emptyset$, the number $m_{i,j}$ of blocks $B\in \cB$ with $|B \cap B_1|=i$ and $|B \cap B_2|=j$
is given in the following table:
\begin{center}
\begin{tabular}{|r|ccccc|} \hline
$i \backslash j$  &  0   &      1      &      2       &      3      &      4      \\  \hline
                     0   &   0   &      0      &     2        &      0     &       1      \\
                    1   &   0   &      0      &     8        &      0     &       0      \\
                    2   &   2   &      8      &     8        &      0     &       0      \\
                    3   &   0   &      0      &     0        &      0     &       0      \\
                    4   &   1   &      0      &     0        &      0     &       0      \\ \hline
\end{tabular}
\end{center}
For each $i\in B_1$, the two blocks meeting $B_1$ only in this point partition $B_2$ into two sets of size two by their 
intersections. We select one of these two to determine two neighbours of $i$ along an octagon as in Figure
\ref{Fig:octagon}. Once $B_1$ and $B_2$ have been selected and this choice of a pair of neighbours
has been made, this determines a quadratic generator for our ideal. We illustrate this with $B_1=\{0,1,2,3\}$
and $B_2=\{4,5,7,8\}$. The resulting polynomial is 
$$  g(\bx) = x_0 x_5 - x_5 x_1 +
x_1 x_7 - x_7 x_3 +
x_3 x_8 - x_8 x_2 +
x_2 x_4 - x_4 x_0 . $$
We leave it to the reader to check that the way in which any block intersects this configuration guarantees
that $g(\bc_B)=0$ for any $B\in \cB$. In fact, just five of these polynomials are needed --- along with $\cG_0$ ---
to generate the ideal.

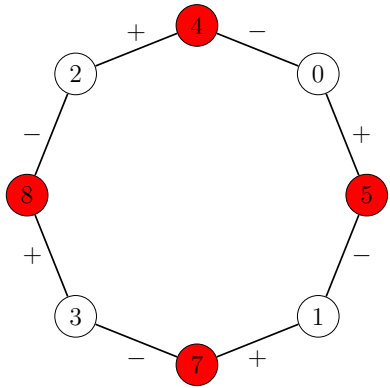
\begin{figure}[htbp]
\label{Fig:octagon}
\begin{center}
\resizebox{2in}{2in}{%
\begin{tikzpicture}[black,node distance=0.5cm,  % Octagon for Ideals of Designs
solidvert/.style={draw, circle,  fill=red, inner sep=3.5pt},
hollowvert/.style={  draw,  circle,  fill=white,  inner sep=3.5pt},
  every loop/.style={min distance=40pt,in=-30,out=90,looseness=20}]
\node[hollowvert] (A1) at (2,2) {$0$};
\node[hollowvert] (A2) at (2,-2) {$1$};
\node[hollowvert] (A3) at (-2,-2) {$3$};
\node[hollowvert] (A4) at (-2,2) {$2$};
\node[solidvert] (B1) at (0,2.8) {$4$};
\node[solidvert] (B2) at (2.8,0) {$5$};
\node[solidvert] (B3) at (0,-2.8) {$7$};
\node[solidvert] (B4) at (-2.8,0) {$8$};
\path (A1) [thick] edge node [above] {$-$} (B1);
\path (A1) [thick] edge node [right] {$+$} (B2);
\path (A2) [thick] edge node [right] {$-$} (B2);
\path (A2) [thick] edge node [below] {$+$} (B3);
\path (A3) [thick] edge node [below] {$-$} (B3);
\path (A3) [thick] edge node [left] {$+$} (B4);
\path (A4) [thick] edge node [left] {$-$} (B4);
\path (A4) [thick] edge node [above] {$+$} (B1);
\end{tikzpicture}}
\end{center}
\caption{Two disjoint blocks $\{0,1,2,3\}$ and $\{4,5,7,8\}$ and the quadratic polynomial
obtained from the pair.}
\end{figure}

%%%%%%%%%%%%%%%%%%%%%%%%%%%%%%%%%%%%%%%%%%
\section{Conclusion}
\label{sec:conclusion}

We have introduced an algebraic approach to the study of $t$-designs which builds on existing machinery tied to 
the space of polynomial functions on blocks. For a design $(X,\cB)$, we introduced the ideal $\cI(\cB)$ and proposed
two parameters $\gamma_1(\cB)$ and $\gamma_2(\cB)$ which we claim capture essential information in the
case of designs where the number of blocks achieves, or is close to, the bound of Ray-Chaudhuri and Wilson. We
prove, among other things, that 
$$ \frac{t+1}{2} \le \gamma_1(\cB) \le \gamma_2(\cB) \le k $$
with the upper bound of $k$ replaced by $t$ in the case of Steiner systems or partial Steiner systems. We
determine the exact value of these parameters for symmetric 2-designs and the Witt designs. By
constructing many triple systems with $\gamma_2(\cB)=k$, we indicate that $\gamma_2(\cB)$ can be larger
than $t$. While we expect the value to be more typically close to $k$, we leave this as an open problem. 

One may also investigate the ideal vanishing on the codewords of an error-correcting code.  In order to compute 
$\gamma_1$ and $\gamma_2$ in such a situation, we have some degree of freedom as these parameters are invariant under affine transformations (provided one is careful with the definition of the trivial ideal). Representing
the codewords of a binary linear $[n,k,d]$ code $C$ by $\pm 1$ vectors in $\re^n$, we see that each dual codeword
$c=[c_1,\ldots, c_n]$ corresponds to an element $f_c(\bx) = -1 +\prod_j x_j^{c_j} $ in the ideal of our code. In
the linear case, $\gamma_1(C)$ is the minimum distance of $C^\bot$ and $\gamma_2(C)$ seems tied to the smallest 
$g$ such that $C^\bot$ is generated by its codewords of weight $g$ or less. So it seems interesting to classify those
codes $C$ for which $\gamma_1(C)=\gamma_2(C)$ as these seem related to tight designs.

\section*{Acknowledgments}

The authors thank Padraig \'{O} Cath\'{a}in, Bill Kantor and Brian Kodalen for useful comments on the work presented here.

\end{document}